\documentclass[12pt]{article}
\usepackage{amsfonts, amsmath, amssymb, amsgen, amsthm, amscd,latexsym,mathrsfs}
\usepackage{color}
\usepackage[all]{xy}

\def\End{\mathop{\rm End}\nolimits}

\def\Id{\mathop{\rm Id}\nolimits}
\def\Im{\mathop{\rm Im}\nolimits}

\def\Hom{\mathop{\rm Hom}\nolimits}

\def\Zb{{\mathbb Z}}

\def\Ac{{\cal A}}
\def\Bc{{\cal B}}

\def\a{\alpha}
\def\b{\beta}
\def\d{\delta}

\def\s{\sigma}

\def\t{\theta}

\def\vp{\varphi}

\def\0b{\bf 0}

\def\ot{\otimes}

\def\ra{\rightarrow}

\def\rt{\cdot}

\def\0D{\Delta^{(0)}}
\def\1D{\Delta^{(1)}}

\newtheorem{theorem}{Theorem}[section]
\newtheorem{remark}[theorem]{Remark}
\newtheorem{proposition}[theorem]{Proposition}
\newtheorem{lemma}[theorem]{Lemma}
\newtheorem{corollary}[theorem]{Corollary}

\newtheorem{example}[theorem]{Example}
\newtheorem{definition}[theorem]{Definition}

\def\ni{\noindent}

\def\build#1_#2^#3{\mathrel{
\mathop{\kern 0pt#1}\limits_{#2}^{#3}}}

\def\odots{\ot\cdots\ot}

 \newcommand{\ie}{{\it i.e.\/}\ }

\def\a{\alpha}
\def\b{\beta}
\def\d{\delta}

\def\s{\sigma}
\def\t{\theta}

\def\vp{\varphi}

\def\ot{\otimes}
\def\part{\partial}

\def\ra{\rightarrow}

\def\lra{\longrightarrow}
\def\text{\hbox}

\def\ot{\otimes}

\def\ra{\rightarrow}

\def\End{\mathop{\rm End}\nolimits}

\def\Hom{\mathop{\rm Hom}\nolimits}

\def\Id{\mathop{\rm Id}\nolimits}

\def\build#1_#2^#3{\mathrel{
\mathop{\kern 0pt#1}\limits_{#2}^{#3}}}

\numberwithin{equation}{section}
\begin{document}
\title{Cyclic homology for Hom-associative algebras}
\author{Mohammad Hassanzadeh, Ilya Shapiro\footnote{Research supported in part by an NSERC Discovery grant 406709}, Serkan S\"utl\"u}

\date{
}

\maketitle

\begin{abstract}
\ni In the present paper we investigate the noncommutative geometry of a class of algebras, called the Hom-associative algebras, whose associativity is twisted by a homomorphism. We define the Hochschild, cyclic, and periodic cyclic homology and cohomology for this class of algebras generalizing these theories from the associative to the Hom-associative setting.
\end{abstract}

\tableofcontents


\section{Introduction}

Starting with a Lie algebraic approach to non-commutative geometry \cite{Wulk96,Wulk99,Wulk99-II}, the guiding motivation behind a ``non-associative geometry" is to extend the non-commutative formalism of the spectral action principle \cite{ChamConnMarc07,ChamConn07,ChamConn08,ChamConn12} to a non-associative framework. More specifically, in an attempt to reformulate the Grand Unified Theories based on $SU(5)$, $SO(10)$ and $E_6$, it is pointed out in \cite{BoylFarn14}, and illustrated in \cite{BoylFarn13}, that in the ordinary approach to physics the basic input is a ``symmetry group", which is associative by nature, whereas in the spectral approach it is an ``algebra", which is not necessarily associative.

On the other hand, the study of the differential geometry of quantum groups \cite{BeggMaji06,MajiOeck99} showed the lack of an associative differential algebra structure on the standard quantum groups, and hence the need for a non-associative geometry (differential geometry with non-associative coordinate algebras) for a full understanding of the geometry of quantum groups. The first step in this direction was taken in \cite{AkraMaji04}, generalizing the twisted cyclic cohomology of \cite{KustMurpTuse03} to the setting of quasialgebras, in order to cover examples motivated by the Poisson geometry \cite{AlekKosmSchwMein02}.

In the present paper we take up a similar analysis from the point of view of a different class of possibly non-associative algebras, called Hom-associative algebras, with the goal of extending the ordinary cyclic homology and cohomology of associative algebras to the non-associative setting.

The Hom-associative algebras first appeared  in  contexts related to physics. The study of $q$-deformations,
based on  deformed derivatives, of Heisenberg algebras, Witt and Virasoro algebras, and the quantum conformal algebras reveals a generalized Lie algebra structure in which the Jacobi identity is deformed by a linear map. These algebras first appeared  in \cite{HartLarsSilv06,LarsSilv05} and they were called Hom-Lie algebras. The Hom-associative
algebras were first introduced  in \cite{MakhSilv08} and were developed in \cite{MakhSilv10,MakhSilv09,Yau09,MakhSilv10-II,Yau10,Yau12}. Briefly, a Hom-associative algebra $\Ac$ satisfies the usual algebra axioms with the associativity condition twisted by an algebra homomorphism $\a:\Ac\rightarrow\Ac$.  More precisely we have $$\a(a)(bc)=(ab)\a(c)$$ for all $a,b,c\in\Ac$.  Thus with $\a=\text{Id}$ we recover the associative algebras as a subclass.

With the goal of extending the formal deformation theory (introduced in \cite{Gers64} for associative algebras, and in \cite{NijeRich66} for Lie algebras) to Hom-associative and Hom-Lie algebras, a cohomology theory for Hom-associative algebras was introduced in \cite{MakhSilv10}. The first and the second cohomology groups of a Hom-associative algebra thus defined were adapted to the deformation theory of Hom-associative algebras, and generalized the Hochschild cohomology of an algebra with coefficients in the algebra itself.


The first and the second cohomology groups of \cite{MakhSilv10-II} were later analyzed more conceptually in \cite{AmmaEjbeMakh11} where the authors defined a Hochschild cohomology for Hom-associative algebras (generalizing the ordinary Hochschild cohomology of an algebra with coefficients in itself) with a Gerstenhaber bracket endowing the differential complex with a graded Lie algebra structure.


The purpose of the present paper is to extend the usual notions of cyclic homology and cohomology for associative algebras to the setting of Hom-associative algebras. The lack of associativity, as the first obstacle on the way to this extension, is partly overcome by restricting our scope to the multiplicative Hom-associative algebras. With this multiplicativity assumption, the presence or absence  of a unit plays a very important role: the classes of naturally occurring examples are very different in flavor. In particular, the multiplicative unital Hom-associative algebras are very close to being associative. On the other hand, in the absence of a unit, one cannot define the Connes' boundary map $B$ \cite{Conn85}, hence we do not have the $(b,B)$-complex interpretation of cyclic (co)homology. Similarly, since we cannot define (co)degeneracy operators, we cannot use the cyclic module approach of \cite{Conn83} to cyclic (co)homology. As a result, we focus only on defining the cyclic (co)homology as a (co)kernel of Hochschild cohomology \cite{Conn85}, and using the bicomplex approach of \cite{Tsyg83}.


This requires a discussion of the Hochschild cohomology with coefficients. We recall that the cyclic homology of an associative algebra $A$ is given by the coinvariants of the Hochschild homology of $A$ with coefficients in $A$ under the cyclic group action, whereas the cyclic cohomology of $A$ is computed by the cyclic invariants of the Hochshild cohomology of $A$ with coefficients in the dual space $A^\ast$. In the case of Hom-associative algebras, it is only the Hom-associative algebra $\Ac$ itself that has been considered as a coefficient space, by which a Hochschild cohomology theory was defined in \cite{AmmaEjbeMakh11}. In the present paper, on the other hand, we define a Hochschild homology theory that can admit the Hom-associative algebra $\Ac$ itself as coefficients, and a Hochschild cohomology theory that can admit $\Ac^\ast$ as a coefficient space.


A rather surprising fact in the Hom-associative setting is that for a Hom-associative algebra $\Ac$, the dual space $\Ac^\ast$ is not an $\Ac$-bimodule (in the sense of \cite[Def. 1.5]{MakhSilv10-II}) via the coregular action. Furthermore, modifying the coregular action by the homomorphism that twists the associativity does not fix this problem. One of the most natural options then is to impose further conditions on the Hom-associative algebra $\Ac$ so that $\Ac^\ast$ becomes an $\Ac$-bimodule. For instance, if $\a\in \End(\Ac)$ is an element of the centroid \cite{Herstein-book}, then $\Ac^\ast$ is an $\Ac$-bimodule. The other option is to introduce a variant of $\Ac^\ast$ as the coefficient space so that it becomes $\Ac^\ast$ in the case when $\Ac$ is associative. This can be achieved by defining
$$\Ac^\circ=\{f\in\Ac^\ast\mid f(x\a(y))=f(\a(xy))=f(\a(x)y)\}$$ as we discuss below. In the case of an associative algebra $A$, we recall from \cite{Conn85} that it is precisely the Hochschild cohomology $H^n(A,A^\ast)$ with coefficients in $A^\ast$ that is equal to the space of the de Rham currents of dimension $n$ when $A=C^\infty(M)$, the algebra of smooth functions on a compact smooth manifold $M$. Hence, in order to capture the correct geometric data, we unify all of these methods in a Hochschild cohomology theory for a Hom-associative algebra $\Ac$ with coefficients in a new object which we call a dual module.


The paper is organized as follows. In Section 2  we review the basics of Hom-associative algebras, in particular we study those having $\a$ in the centroid. We note for a Hom-associative algebra $\Ac$ that $\Ac^\ast$ is not necessarily an $\Ac$-bimodule, and thus we investigate the algebraic dual of a module over a Hom-associative algebra from the representation theory point of view. Finally, we recall the basics of Hochschild and cyclic homology and cohomology for associative algebras.
In Section 3 we introduce Hochschild homology of  $\Ac$ with coefficients in a bimodule $V$. Then for $V=\Ac$ we introduce the cyclic group action on the Hochschild complex, and define the cyclic homology of a multiplicative Hom-associative algebra $\Ac$. Finally, we use the bicomplex method to define the cyclic and the periodic homologies of $\Ac$, and we show that the equivalency of the two definitions. In section 4 we define the Hochschild cohomology of $\Ac$ with coefficients in a new object $V$ which we call a dual module. We define the cyclic group action on the Hochschild complex of $\Ac$, with coefficients in the dual module $\Ac^*$, and hence the cyclic cohomology of $\Ac$. Similar to the homology case, we use the bicomplex method to define the cyclic and periodic cyclic cohomologies of $\Ac$, and we show that the two definitions agree.

\medskip
\textbf{Notation}: Throughout the paper all algebras are over a field.  We reserve the font $A$ for associative algebras whereas for Hom-associative algebras we use $\Ac$. All tensor products are over the field of the algebra in question.

\section{Preliminaries}

\subsection{Hom-associative algebras}

In this subsection we recall the definition of a Hom-associative algebra, and in addition to basic examples, we study characterization results on the unitality and the embedding properties of Hom-associative algebras.


Hom-associative algebra structures were introduced recently in \cite{MakhSilv10}, and then the theory was developed further to include Hom-coalgebras, Hom-bialgebras and Hom-Hopf algebras in \cite{MakhSilv09,Yau09,MakhSilv10,Yau10}, and even Hom-quantum groups in \cite{Yau12}. We also refer to \cite{HartLarsSilv06,MakhSilv09} for the Lie-counterpart of the theory.


Let us recall from \cite{Yau09} that a Hom-associative algebra is a triple $(\Ac, \mu, \alpha)$ consisting of a vector space $\Ac$ over a field $k$, and $k$-linear maps $\mu: \Ac\ot \Ac\lra \Ac$ that we denote by $\mu(a,b)=:ab$, and $\alpha: \Ac\longrightarrow \Ac$ satisfying the Hom-associativity condition
\begin{equation}\label{aux-Hom-assoc-cond}
\a(a)(bc)=(ab)\alpha(c),
\end{equation}
for any $a,b,c\in \Ac$. Hom-associative algebras generalize associative algebras in the sense that any associative algebra $\Ac$ is a Hom-associative algebra with $\alpha:=\Id_\Ac$, the identity map. A Hom-associative algebra $\Ac$ is called multiplicative if for any $a,b\in\Ac$
\begin{equation*}
\alpha(ab) = \a(a)\a(b).
\end{equation*}
\begin{example}\label{twist}
{\rm
Let $\Ac$ be any associative algebra with multiplication $\mu:\Ac\ot \Ac\lra \Ac$, and let $\alpha: \Ac\longrightarrow \Ac$ be an algebra map. Then for $\mu_\a=\alpha \circ \mu:\Ac\lra\Ac$, the triple  $(\Ac, \mu_{\alpha}, \alpha)$ is a multiplicative Hom-associative algebra \cite{Yau09}.
}
\end{example}

Throughout the paper, by a Hom-associative algebra we mean a multiplicative Hom-associative algebra unless otherwise is stated.


A Hom-associative algebra $\Ac$ is called unital if there is an element $1\in \Ac$ such that $1a=a=a1$. Unital Hom-associative algebras first appeared  in \cite{FregGohr09}, see also \cite{Gohr10}, and a classification of unital Hom-associative algebras is given in \cite{FregGohr09}. If a Hom-associative algebra $\Ac$ is unital, then for any $a,b\in\Ac$,
\begin{equation}\label{aux-almost-unital-Hom-alg}
\a(a)b=a\a(b)=\a(ab).
\end{equation}
It follows from the multiplicativity of $\Ac$ that $$\a(a)b=a\a(b)=\a(ab)=\a(a)\a(b)=\a^2(a)b,$$ from which we conclude, by plugging in $b=1$, that
\begin{equation*}
\a = \a^2.
\end{equation*}

More precisely, we have the following characterization result.

\begin{lemma}\label{lemma-unital}
Let $(\Ac,\mu,\alpha,1)$ be a unital Hom-associative algebra. Then $\Ac\cong A_1\oplus A_2$ as algebras, where $A_1$ is a unital associative algebra, and $A_2$ is a unital (not necessarily associative) algebra.  Furthermore, $\a:\Ac\lra \Ac$ is given by $\alpha(a_1+a_2)=a_1.$  Conversely, for any unital associative algebra $A_1$ and a unital (not necessarily associative) algebra $A_2$, $A_1\oplus A_2$ is a unital Hom-associative algebra with $\alpha:A_1\oplus A_2\lra A_1\oplus A_2$ being the projection onto $A_1$.
\end{lemma}

\begin{proof}
It is shown in \cite{FregGohr09} that for a unital, Hom-associative algebra $\Ac$, the map $\alpha:\Ac\lra \Ac$ is given by the left multiplication $L_x:\Ac\lra \Ac$, $a\mapsto xa$ with a central element $x\in \Ac$. We then conclude, using the multiplicativity $x(ab)=(xa)(xb)$, that $x^2=x$. Letting $y=1-x$, we observe that $y^2=y$, that is $y\in \Ac$ is also central, and $\Ac\cong \Ac x\oplus \Ac y$ as algebras. We note that $A_1:=\Ac x$ and $A_2:=\Ac y$ are both unital with units $x\in \Ac x$ and $y\in \Ac y$ respectively. Finally $\alpha$-associativity of $\Ac$ implies the associativity of $A_1$, since for $xa, xb, xc\in A_1$ we have $(xa)(xbxc)=xa(bc)=\alpha(a)(bc)=(ab)\alpha(c)=(ab)xc=(xaxb)(xc)$, and $\a(xa+yb)=x(xa+yb)=xa$, that is $\a:\Ac\lra \Ac$ is the projection onto $A_1$. The converse statement, on the other hand, is straightforward.
\end{proof}

\begin{example}{\rm
Let $A$ be an associative algebra, and $\a:A\lra A$ an algebra endomorphism.  Assume that $\alpha^2=\a$.  Let $x\star y=\a(xy)$, so that $(\Ac,\star,\a)$ is a multiplicative Hom-associative algebra.  Then as algebras $(\Ac,\star,\a)=K\oplus B$ where the multiplication on $K$ is $0$, the multiplication on $B$ is associative, and the associativity endomorphism is the projection $proj_B$ onto $B$. To see this we let $K=ker(\a)$, so that $K$ is an ideal of $\Ac$, and let $B=\Im(\a)$.  Then as vector spaces $\Ac=K\oplus B$ and the associativity endomorphism is $proj_B$.  Since $K$ is an ideal, $(k+b)\star (k'+b')=bb'$.
}
\end{example}

\begin{example}\label{unital-example}
{\rm
Let $\Ac$ be a two dimensional vector space over a field $k$ with a basis $\{e_1, e_2\}$. Let the multiplication $\mu: \Ac\ot \Ac\longrightarrow \Ac$ be given by
\begin{equation*}
e_ie_j = \left\{\begin{array}{cc}
                    e_1, & {\rm if}\,\,(i, j)= (1,1) \\
                    e_2 & {\rm if}\,\,(i, j)\neq (1,1).
                  \end{array}
\right.
\end{equation*}
Then via the map
\begin{equation*}
\alpha:\Ac\lra \Ac,\qquad \alpha(e_1)=e_1-e_2,\quad \alpha(e_2)=0,
\end{equation*}
the triple $(\Ac, \mu, \alpha)$ is a Hom-associative algebra with the unit $1:=e_1$. Furthermore we have $\a^2=\a$.  In view of the Lemma \ref{lemma-unital} we see that $\Ac=k(e_1-e_2)\oplus ke_2$.
}
\end{example}

We will need the notion of a bimodule over a Hom-associative algebra to serve as a coefficient space for the Hochschild homology and cohomology. Hence we recall it from \cite[Def. 1.5]{MakhSilv10}.

\begin{definition}
Let $(\Ac,\mu,\a)$ be a Hom-associative algebra. Then a linear space $V$ equipped with $\rt:\Ac\ot V \lra V$, $a\ot v\mapsto a\rt v$, and $\b:V\lra V$, is called a left $\Ac$-module if the diagram
\begin{align*}
\xymatrix{
\Ac\ot \Ac \ot V \ar[r]^{\,\,\,\,\,\mu\ot \b} \ar[d]_{\a\ot\, \rt} & A\ot V \ar[d]^{\rt}\\
\Ac\ot V \ar[r]_{\rt} & V
}
\end{align*}
is commutative, i.e.,
\begin{equation}\label{aux-Hom-module}
(ab)\cdot \b(v) = \a(a)\cdot (b\cdot v),
\end{equation}
for any $a,b\in \Ac$ and any $v\in V$.
\end{definition}

Similarly, $(V,\b)$ is called a right $\Ac$-module if
\begin{equation*}
\beta(v)\cdot (ab)= (v\cdot a)\cdot \a(b).
\end{equation*}

\begin{example}\label{ex-A-is-A-module}
{\rm
Any Hom-associative algebra $(\Ac,\mu,\a)$ is both a left and a right module over itself by $\beta=\a$, \cite[Rk. 1.6]{MakhSilv10}.
}
\end{example}

In order for defining the cyclic cohomology theory for Hom-associative algebras, we shall need the cyclic invariant subcomplex of the Hochschild complex of $\Ac$, with coefficients in $\Ac^\ast$. However, given a Hom-associative algebra $\Ac$, the algebraic dual $\Ac^\ast$ is not necessarily an $\Ac$-module via the coregular actions, $(a\cdot f)(b)=f(ba)$ or $(f\cdot a)(b)=f(ab)$, or their $\a$-twisted versions $(a\cdot f)(b)=f(b\a(a))$ or $(f\cdot a)(b)=f(\a(a)b)$. Then one is forced either to restrict $\Ac^\ast$ into a subspace, that we denote by $\Ac^\circ$ below, which can be an $\Ac$-bimodule, or impose further conditions on $\Ac$ so that $\Ac^\ast$ becomes an $\Ac$-bimodule. In the former case we have the following.

\begin{lemma}
Given a Hom-associative algebra $(\Ac,\mu,\a)$, the pair $(\Ac^\circ,\Id_{\Ac^\ast})$ where
\begin{equation}\label{aux-A-circ}
\Ac^\circ=\{f\in\Ac^\ast\mid f(x\a(y))=f(\a(xy))=f(\a(x)y)\},
\end{equation}
is a left $\Ac$-module via
\begin{equation*}
(a\cdot f)(b)=f(b\a(a)),
\end{equation*}
for any $a,b\in\Ac$, and any $f\in \Ac^\circ$.
\end{lemma}

\begin{proof}
For $a\in\Ac$ and $f\in \Ac^\circ$, we verify first that $a\cdot f\in \Ac^\circ$. Indeed,
\begin{align*}
 a\cdot f (x\a(y)) &= f((x\a(y))\a(a))=f(\a(x)(\a(y)a))=f(\a(x(\a(y)a)))\\
&=f(\a(x)\a(\a(y)a))=f(\a^2(x)(\a(y)a))\\&=f((\a(x)\a(y))\a(a)) = a\cdot f (\a(xy)).
\end{align*}
Similarly,
\begin{align*}
 a\cdot f (x\a(y)) &=f(\a(xy)\a(a)) = f((xy)\a(a))=f(\a(x)(ya))=f(\a(x)\a(ya))\\
&= f(\a^2(x)(ya))=f((\a(x)y)\a(a))=a\cdot f(\a(x)y).
\end{align*}
As for the left module condition, we have
\begin{align*}
(ab)\cdot f (x) &= f(x\a(ab))=f(\a(x)(ab))=f((xa)\a(b))=(b\cdot f)(\a(xa))\\
&= (b\cdot f)(x\a^2(a))=\a(a)\cdot (b\cdot f)(x).
\end{align*}

\end{proof}

In the same fashion, $\Ac^\circ$ is a right $\Ac$-module by $(f\cdot a)(b)=f(\a(a)b)$. In comparison to $\Ac^\ast$, we note that in case $\Ac$ is unital, or is associative, then $\Ac^\circ=\Ac^*$.

We shall also need the notion of a bimodule for the cohomological considerations.

\begin{definition}
Let $(\Ac,\mu,\a)$ be a Hom-associative algebra, and $(V,\b)$ be a left and a right $\Ac$-module. Then $V$ is called an $\Ac$-bimodule if
\begin{equation}\label{aux-A-bimodule}
\a(a)\cdot (v \cdot b) = (a \cdot v)\cdot \a(b)
\end{equation}
for any $a,b\in \Ac$, and any $v \in V$.
\end{definition}

\begin{example}\label{ex-A-is-A-bimodule}
{\rm
Any Hom-associative algebra $(\Ac,\mu,\a)$, the pair $(\Ac,\a)$ is a bimodule over itself.
}
\end{example}

\begin{proposition}
Given a Hom-associative algebra $(\Ac,\mu,\a)$, the subspace $\Ac^\circ\subseteq \Ac^\ast$ is an $\Ac$-bimodule.
\end{proposition}

\begin{proof}
We have already seen that $A^{\circ}$ is a left and right $\Ac$-module. Let us show that the bimodule compatibility \eqref{aux-A-bimodule} is satisfied. Indeed, for any $a,b,x\in\Ac$, and any $f\in \Ac^\circ$,
\begin{align*}
\a(a)\cdot (f\cdot b) (x) &= f(\a(b)(x\a^2(a)))=f(\a(b)\a(x\a^2(a)))=f(\a^2(b)(x\a^2(a)))\\
&= f(\a^2(b)\a(x\a^2(a))) = f(\a^3(b)(x\a^2(a))) = f((\a^2(b)x)\a^3(a))\\
&= f(\a(\a^2(b)x)\a^2(a))=f((\a^2(b)x)\a(a))= ((a\cdot f)\cdot \a(b))(x).
\end{align*}
\end{proof}

In order for $\Ac^\ast$ to be an $\Ac$-bimodule, a second option is to impose further conditions on $\Ac$. To this end, we need the following class of Hom-associative algebras.
For a not necessarily multiplicative Hom-associative algebra $(\Ac,\mu,\a)$, we say $\a\in\End(\Ac)$ is an element of the centroid if
\begin{equation}\label{aux-unital-Hom-alg}
  \a(x)y=x\a(y)=\a(xy),
\end{equation}
for any $x,y\in \Ac$. We refer the reader to \cite{Herstein-book} for more information on the centroid of a ring, and \cite{BenaMakh14} for the centroid of a Lie algebra in the concept of Hom-Lie algebras.
We note that for any unital Hom-associative algebra $(\Ac,\a)$, $\a\in\End(\Ac)$ is an element of the centroid. On the other hand, if for a Hom-associative algebra $(\Ac,\a)$ the mapping $\a\in\End(\Ac)$ is an element of the centroid and $\a(1)=1$, then $\a=\Id$, and hence $\Ac$ is associative.

\begin{proposition}
Let $(\Ac,\mu,\a)$ be a multiplicative  Hom-associative algebra where $\a$ is an element of the centroid. Then $(\Ac^*,\a^\ast)$ is an  $\Ac$-bimodule via the coregular actions.
\end{proposition}

\begin{proof}
Since $\a$ is an element of the centroid we have
\begin{align*}
 ((aa')\cdot\a^\ast(f))(b)&=\a^\ast(f)(b(aa'))=f(\a(b(aa')))
  =f(\a(b)(aa'))\\&=f((ba)\a(a'))=f(\a(ba)a')
  =f((b\a(a))a')\\&=(a'\cdot f)(b\a(a))=(\a(a)\cdot(a'\cdot f))(b),
\end{align*}
that is $(\Ac^\ast,\a^\ast)$ is a left $\Ac$-module. Similarly
\begin{align*}
  (\a^\ast(f)(ab))(x)&= \a^\ast(f)((ab)x)=f(\a(ab)x)=f(a\a(b)x)\\
  &= (f\cdot a)(\a(b)x)=((f\cdot a)\cdot \a(b))(x),
\end{align*}
proving that $(\Ac^\ast,\a^\ast)$ is a right $\Ac$-module. We next observe that $\Ac^*$ is an $\Ac$-bimodule. Indeed,
\begin{align*}
 (\a(a)\cdot(f\cdot b))(c) &= (f\cdot b)(c\a(a)) =  f(b(c\a(a)))
 =f(b(\a(ca))\\&=f(\a(b)(ca))=f((bc)\a(a))=f(\a(bc)a)\\
 & =f((\a(b)c)a)=(a\cdot f)(\a(b)c)    =((a\cdot f)\cdot \a(b))(c).
\end{align*}
\end{proof}

It turns out that the  Hom-associative algebras whose $\a$ is an element of the centroid can be characterized by being embeddable into unital ones.

\begin{proposition}\label{prop-embed}
Let  $(\Ac,\mu,\a)$ be a Hom-associative algebra. The map $\a\in\End(\Ac)$ is an element of the centroid if and only if $\Ac$ can be embedded into a unital Hom-associative algebra.
\end{proposition}

\begin{proof}
If $\Ac$ can be embedded into a unital Hom-associative algebra, then by \cite{FregGohr09} the equation \eqref{aux-unital-Hom-alg} holds.

Conversely, assuming the equation \eqref{aux-unital-Hom-alg}, we let $\Bc=k[\a]\oplus \Ac$ with multiplication $(p(\a)+a)(q(\a)+b)=p(\a)q(\a)+(q(\a)(a)+p(\a)(b)+ab)$, where $q(\a)(a)$ is the element of $\Ac$ obtained by interpreting $\a\in k[\a]$ as the linear endomorphism $\a$ that is part of the Hom-associative algebra structure of $\Ac$.  Define $(\Bc,\mu,\beta)$ by setting $\beta|_{\Ac}=\a$ and $\beta|_{k[\a]}=m_{\a}$, the multiplication by $\alpha$.
Note that \eqref{aux-unital-Hom-alg} implies that $\a\a^k(ab)=\a(a)\a^k(b)=\a^k(a)\a(b)$ for $k\geq 0$ and so $\a s(\a)(ab)=\a(a)(s(\a)(b))=(s(\a)(a))\a(b)$.  With this in mind we check that
\begin{align*}
\beta&(p(\a)+a)((q(\a)+b)(s(\a)+c))\\
&=(\a(p(\a))+\a(a))(q(\a)s(\a)+q(\a)(c)+s(\a)(b)+bc)\\
&=\a p(\a)q(\a)s(\a)+(\a p(\a))(q(\a)(c))+(\a p(\a))(s(\a)(b))+(\a p(\a))(bc)\\&\quad+(q(\a)s(\a))(\a(a))+\a(a)(q(\a)(c))+\a(a)(s(\a)(b))+\a(a)(bc)\\&=p(\a)q(\a)\a s(\a)+\a s(\a)(q(\a)(a))+\a s(\a)(p(\a)(b))+\a s(\a)(ab)\\&\quad+(p(\a)q(\a))(\a(c))+(q(\a)(a))\a(c)+(p(\a)(b))\a(c)+(ab)\a(c)\\&=(p(\a)q(\a)+q(\a)(a)+p(\a)(b)+ab)(\a s(\a)+\a(c))\\&=((p(\a)+a)(q(\a)+b))\beta(s(\a)+c).
\end{align*}
Hence, $(\Bc,\mu,\beta)$ is a unital Hom-associative algebra with a subalgebra $\Ac$.
\end{proof}

Furthermore, the following observation characterizes the multiplicativity of the Hom-associative algebra $\Bc$.

\begin{lemma}
Let  $\Bc$ be as in Proposition \ref{prop-embed}. Then $\Bc$ is multiplicative if and only if  $\alpha^2=\alpha $ on $\Ac$.
\end{lemma}


We unify the above two approaches to the problem of the representation of a Hom-associative algebra $\Ac$ on $\Ac^\ast$ in the following manner.

\begin{definition}
  Let $(\Ac, \a)$ be a Hom-associative algebra. A vector space $V$ is called a dual left $\Ac$-module if there are linear maps $\cdot: \Ac\ot V\longrightarrow V$, and  $\b: V\longrightarrow V$ where
  \begin{equation}
    a\cdot (\a(b)\cdot v)=\b((ab)\cdot v).
  \end{equation}
\end{definition}
Similarly, $V$ is called a dual right $\Ac$-module if $v\cdot (\a(a))\cdot b= \b(v\cdot (ab))$. Finally, we call $V$ a dual $\Ac$-bimodule if $\a(a)\cdot (v\cdot b)=(a\cdot v)\cdot \a(b).$

The definition is just another Hom-associative generalization of the notion of bimodule. In case $\Ac$ is associative, the above definition coincides with the definition of a bimodule over an associative algebra.

The next observation justifies our choice of dual module being one which is compatible with $\Ac^\ast$ as a coefficient space, rather than $\Ac$. In this language the definition \cite[Def. 1.5]{MakhSilv10} of module over a Hom-associative algebra is one which is compatible with $\Ac$, which served as the coefficient space in the Hochschild cohomology theory of \cite{AmmaEjbeMakh11}.

\begin{lemma}\label{property}
Let $(\Ac, \a)$ be a Hom-associative algebra, and $(V, \beta)$ an $\Ac$-bimodule. Then the algebraic dual $V^*$ is a dual $\Ac$-bimodule.
\end{lemma}

\begin{proof}
We first note that
\begin{align*}
(a\cdot (\a(b)\cdot f))(v)&= f((v\cdot a)\cdot\a(b))
=f((\beta(v)\cdot ab)\\&= ((ab)\cdot f)(\b(v))=\b((ab)\cdot f)(v).
\end{align*}
Therefore $V^*$ is a dual left $\Ac$-module. That it is a dual right $\Ac$-module, as well as a dual $\Ac$-bimodule is similar.
\end{proof}

We conclude this subsection recalling from \cite{MakhSilv09} the morphisms of Hom-associative algebras. Let $(\Ac, \a_{\Ac})$ and $(\Bc, \a_{\Bc})$ be two Hom-associative algebras. Then a linear map $f: (\Ac, \a_{\Ac})\longrightarrow (\Bc, \a_{\Bc})$ is called a morphism of Hom-associative algebras if it is an algebra morphism and
$f(\a_{\Ac}(x))=\a_{\Bc}(f(x))$ for $x,y\in \Ac.$

\section{Homology of Hom-associative algebras}

In this section we introduce Hochschild, cyclic and periodic cyclic homologies for Hom-associative algebras. We first define Hochschild homology of a Hom-associative algebra $\Ac$ with coefficients in an $\Ac$-bimodule. Then, using $\Ac$ itself as the coefficient space via \cite[Rk. 1.6]{MakhSilv10}, we define the cyclic homology of $\Ac$ via the cyclic coinvariant of the Hochschild complex of $\Ac$.

\subsection{Hochschild homology of a Hom-associative algebra}

In this subsection we define the Hochschild homology of a Hom-associative algebra with coefficients in an $\Ac$-bimodule.

\begin{theorem}
Let $(\Ac,\mu, \alpha)$ be a Hom-associative algebra, and $(V, \beta)$ be an $\Ac$-bimodule such that
$$\b(v\cdot a) = \b(v)\cdot \a(a) \quad \text{and} \quad \b(a\cdot v)=\a(a)\cdot \b(v).$$ Then
\begin{equation*}
C^{Hom}_\ast(\Ac, V)=\bigoplus_{n\geq 0}C^{Hom}_n(\Ac, V),\qquad C^{Hom}_n(\Ac, V):=V\ot \Ac^{\ot n},
\end{equation*}
with the face maps
\begin{align*}
&\d_0(v\ot a_1\ot \cdots \ot a_{n})= v \cdot a_1 \ot \a(a_2)\ot \cdots \ot \a(a_{n})\\
&\d_i(v\ot a_1\ot \cdots \ot a_{n})=\b(v)\ot \a(a_1) \cdots \ot a_i a_{i+1}\ot \cdots\ot \alpha(a_{n}), ~~ 1\leq i \leq n-1\\
&\d_{n}(v\ot a_1\ot \cdots \ot a_n)= a_{n} \cdot v \ot \a(a_1)\ot  \cdots\ot \a(a_{n-1})
  \end{align*}
is a presimplicial module.
\end{theorem}

\begin{proof}
We will show that $\delta_i \delta_j= \delta_{j-1} \delta_{i}$ for $0\leq i< j \leq n$ \cite[Def. 1.0.6]{Loday-book}.
The equality $\d_0\d_0=\d_0\d_1$ follows at once from the right $\Ac$-module compatibility $\b(v)\cdot (a_1a_2)= (v \cdot a_1)\cdot \a(a_2)$. We next observe that $\d_0 \d_j=\d_{j-1}\d_0$ for $j> 1$. Indeed,
  \begin{align*}
    \d_0\d_j(v\ot a_1&\odots a_n)\\
    &=\d_0(\b(v)\ot \a(a_1)\odots a_j a_{j+1}\odots \a(a_n)) \\
    &=\b(v)\cdot\a(a_1)\ot \a^2(a_2)\odots \a(a_j a_{j+1})\odots  \a^2(a_n) \\
    &=\b(v\cdot a_1)\ot \a^2(a_2)\odots \a(a_j )\a(a_{j+1})\odots  \a^2(a_n) \\
    &=\d_{j-1}(v\cdot a_1\ot \a(a_2)\odots \a(a_n)) \\
    &=\d_{j-1}\d_0(v\ot a_1\odots a_n).
  \end{align*}
The equality $\d_i\d_n=\d_{n-1}\d_i$ follows similarly. Let us finally show that $\d_0\d_n=\d_{n-1}\d_0$. We have,
  \begin{align*}
    \d_0\d_n(v\ot a_1&\odots a_n)\\
    &=\d_0(a_n\cdot v\ot \a(a_1)\odots \a(a_{n-1}))\\
    &=(a_n\cdot v )\cdot \a(a_1)\ot \a^2(a_2)\odots \a^2(a_{n-1})\\
    &=\a(a_n)\cdot (v\cdot a_1)\ot \a^2(a_2)\odots \a^2(a_{n-1})\\
    &=\d_{n-1}(v\cdot a_1\ot \a(a_2)\odots \a(a_n))\\
    &=\d_{n-1}\d_0(v\ot a_1\odots a_n).
  \end{align*}
\end{proof}

As a result $(C^{Hom}_\ast(\Ac, V),b)$ becomes a differential complex with
\begin{align*}
b=\sum_{i=0}^n (-1)^i \d_i&:C^{Hom}_n(\Ac, V)\longrightarrow C^{Hom}_{n-1}(\Ac, V),\\
b(v\ot a_1\odots a_n)&=v\cdot a_1\ot \a(a_2)\odots \a(a_n)\\
&\quad+\sum_{i=1}^{n-1} (-1)^i \b(v)\ot \a(a_1)\odots a_ia_{i+1} \odots \a(a_n)\\
&\quad+(-1)^n a_n\cdot v\ot \a(a_1)\odots \a(a_{n-1}).
\end{align*}
We denote the homology of the complex $(C^{Hom}_\ast(\Ac, V),b)$ by $H^{Hom}_\ast(\Ac, V)$ and call it the Hochschild homology of the Hom-associative algebra $\Ac$ with coefficients in the $\Ac$-bimodule $V$.

\begin{example}{\rm
If the Hom-associative algebra $(\Ac, \a)$ is associative, i.e., $\a=\Id$, then $H^{Hom}_\ast(\Ac, V)$ is the ordinary Hochschild homology of an associative algebra.
  }
\end{example}

\begin{example}{\rm
  If $\Ac$ is a  Hom-associative algebra and $M$ an $\Ac$-bimodule, then
\begin{equation*}
  H_0(\Ac,M)= \frac{M}{[M,A]}.
\end{equation*}
}\end{example}

\begin{example}{\rm
Let $\Ac$ be the Hom-associative algebra in Example  \ref{unital-example}. Then
  $HH_1(\Ac)=H_1(\Ac, \Ac)$ is one dimensional.
}
\end{example}

\begin{remark}\label{remark-A-alpha}
{\rm
Let $A$ be an associative algebra with the multiplication map $\mu:A\ot A\lra A$, and $\alpha\in\End(A)$. Then for $\Ac_\a = A$, and the multiplication $\mu_\a=\alpha \circ \mu:\Ac\ot\Ac\lra\Ac$, the triple  $(\Ac_\a, \mu_{\alpha}, \alpha)$ is a multiplicative Hom-associative algebra, \cite{Yau09}. It is easy to see that 
\begin{equation*}
(C^{Hom}_\ast(\Ac_\alpha, \Ac_\alpha), d^{Hom})=(C_\ast(A, A), \alpha d=d\alpha)
\end{equation*}
as complexes, where the latter is the usual Hochschild homology complex of the associative algebra $A$ with the modified differential, via composition with $\alpha$ which acts on $A^{\otimes n}$ by $\alpha(a_1\odots a_n)=\alpha(a_1)\odots \alpha(a_n)$.
}
\end{remark}

An immediate consequence is the following lemma.

\begin{lemma}
Let $\Ac_\a$ be the Hom-associative algebra of Remark \ref{remark-A-alpha} such that $\alpha\in\End(\Ac_\a)$ is an isomorphism. Then, 
\begin{equation*}
HH^{Hom}_\ast(\Ac_\alpha)=HH_\ast(A).
\end{equation*}
\end{lemma}

\begin{proof}
The claim follows from the observation that if $\alpha:A\lra A$ is an isomorphism, then $\ker(\alpha d)=\ker(d)$ and $\Im(d\alpha)=\Im(d)$.
\end{proof}

Furthermore, it follows from \cite[Corol. 8]{HellMakhSilv14} (see also \cite{Gohr10}) that any multiplicative Hom-associative algebra $(\Ac,\a)$ such that $\alpha\in\End(\Ac)$ is an isomorphism, is of the form $\Ac_\alpha$ for the associative algebra $(\Ac,\alpha^{-1}\circ\mu)$. As a result, we conclude the following reduction of the Hom-Hochschild homology of $\Ac$ to the Hochschild homology of $A$.

\begin{corollary}
If $(\Ac,\mu,\alpha)$ is a multiplicative Hom-associative algebra such that $\alpha:\Ac\lra \Ac$ is an isomorphism, then 
\begin{equation*}
HH^{Hom}_\ast(\Ac)=HH_\ast(A),
\end{equation*}
where $A$ is the associative algebra $(\Ac,\alpha^{-1}\circ\mu)$.
\end{corollary}

\begin{remark}
{\rm
  For associative algebras, it is known that the Hochschild homology of direct sum algebras is a direct sum of Hochschild homologies; more precisely, we have $HH_n(A_1\oplus A_2)\cong HH_n(A_1)\oplus HH_n(A_2).$ However a quick calculation shows that this is not the case for Hom-associative algebras even in the very special case of a unital multiplicative Hom-associative algebra.  Let us recall Example \ref{unital-example}, i.e., $\Ac=k(e_1-e_2)\oplus ke_2=k_1\oplus k_2$ with $\a=proj_{k_1}$.  Thus $\Ac=k_1\oplus k_2$ with $\a_1=Id$ and $\a_2=0$.  We have $HH^{Hom}_i(k_1)=HH_i(k_1)$ and the latter vanishes except for $HH_0(k_1)=k$.  On the other hand $HH^{Hom}_i(k_2)=k$ for all $i$.  We observe that we have a direct sum decomposition of complexes $$C_\ast^{Hom}(\Ac,\Ac)=C^{\leq 2}_\ast\oplus C^{\geq 3}_\ast$$ where $C^{\leq 2}_\ast$ consists of tensors with at most two from $k_2$ and $C^{\geq 3}_\ast$ consists of tensors with at least three from $k_2$.  We remark that the differential on $C^{\geq 3}_\ast$ is identically zero.  This shows that in the higher degrees $HH^{Hom}_\ast(\Ac)$ is much larger than $HH_\ast^{Hom}(k_1)\oplus HH_\ast^{Hom}(k_2)=k$.
}
\end{remark}

\subsection{Cyclic homology of a Hom-associative algebra}

In this subsection we introduce cyclic homology for Hom-associative algebras through two different but equivalent methods.

\subsubsection{Cyclic homology as a cokernel}

We construct the cyclic homology of a Hom-associative algebra $\Ac$ out of the Hochschild homology of $\Ac$ with coefficients in $\Ac$ by factoring out the cyclic group action. Let us denote the Hochschild complex $C^{Hom}_\ast(\Ac,\Ac)$ simply by $C^{Hom}_\ast(\Ac)$. We note that in this case we have
\begin{align}\label{aux-b-Hochschild}
\begin{split}
 b(a_0\odots a_n) &= \sum_{i=0}^{n-1} (-1)^i \a(a_0)\odots a_ia_{i+1} \odots \a(a_n) \\
&\quad+  a_n a_0\ot \a(a_1)\odots \a(a_{n-1}).
\end{split}
\end{align}
Given a Hom-associative algebra $(\Ac,\mu,\a)$ we define the action of the cyclic group $\Zb/(n+1)\Zb$ on $\Ac^{\ot\,n+1}$ as in \cite[Subsect. 2.1.10]{Loday-book}:
\begin{equation}\label{aux-cyclic-group-action}
t_n(a_0\odots a_n) = (-1)^na_n\ot a_0\odots a_{n-1},
\end{equation}
and obtain the cokernel $C_n^{Hom,\,\lambda}(\Ac)= C_n^{Hom}(\Ac)/(\Id-t)$. Next we let
\begin{equation*}
b': C^{Hom}_n(\Ac)\longrightarrow C^{Hom}_{n-1}(\Ac)
\end{equation*}
be defined by
\begin{equation}\label{aux-b-prime}
b'(a_0\odots a_n)= \sum_{i=0}^{n-1}(-1)^i \a(a_0)\ot \a(a_1)\odots a_ia_{i+1} \odots \a(a_n).
\end{equation}
For simplicity we set $t_n=t$.
\begin{lemma}\label{lemma-b-b-prime-t}
The operators \eqref{aux-b-Hochschild}, \eqref{aux-b-prime} and \eqref{aux-cyclic-group-action} satisfy the equality
\begin{equation*}
(\Id - t_{n-1})b' = b(\Id - t_n).
\end{equation*}
\end{lemma}

\begin{proof}
We first observe that $\d_0t=(-1)^n\d_n$. Indeed,
\begin{align*}
\d_0t_n(a_0\odots a_n) &= (-1)^n\d_0(a_n\ot a_0\odots a_{n-1})\\
&= (-1)^n a_na_0 \ot\a(a_1)\odots \a(a_{n-1})\\ &= (-1)^n\d_n(a_0\odots a_n).
\end{align*}
We next observe that $\d_it_n=-t_{n-1}\d_{i-1}$ for $0<i\leq n$. Indeed,
\begin{align*}
 \d_it_n (a_0\odots a_n) &= (-1)^n\d_i(a_n\ot a_0\odots a_{n-1})\\
&= (-1)^n\a(a_n)\ot \a(a_0)\odots a_{i-1}a_i\odots  \a(a_{n-1}) \\
&= -(-1)^{n-1}t_{n-1}(\a(a_0)\odots a_{i-1}a_i\odots  \a(a_n)) \\
&= -t_{n-1}\d_{i-1}(a_0\odots a_n).
\end{align*}
\end{proof}

As a result $b:C^{Hom}_n(\Ac)\lra C^{Hom}_{n+1}(\Ac)$ is well-defined on $C^{Hom}_\ast(\Ac)/\Im(\Id-t)$, and hence $(C^{Hom,\lambda}_\ast(\Ac), b)$ is a differential complex. We call the homology of the complex
$$
\begin{CD}
C_0^{Hom,\lambda}(\Ac) @<b<< C_1^{Hom,\lambda}(\Ac) @<b<< C_2^{Hom,\lambda}(\Ac) @<b<< C_3^{Hom,\lambda}(\Ac) \ldots
\end{CD}\\
$$
the cyclic homology of the Hom-associative algebra $\Ac$, and we denote it by $HC^{Hom,\lambda}_\ast(\Ac)$.

\begin{example}
  {\rm
If the Hom-associative algebra $(\Ac, \a)$ is associative, \ie $\a=\Id$, then $HC^{Hom,\lambda}_\ast(\Ac)$ is the ordinary cyclic homology of an associative algebra.
  }
\end{example}

\begin{example}{\rm
    For any   multiplicative Hom-associative algebra $\Ac$ we have
\begin{equation*}
HC_0(\Ac)= HH_0(\Ac)=\frac{\Ac}{[\Ac,\Ac]}.
\end{equation*}
    }
  \end{example}

\begin{example}{\rm
For the Hom-associative algebra $\Ac$ of Example  \ref{unital-example} we have
  \begin{equation*}
  HC^{Hom,\lambda}_1(\Ac)=0.
  \end{equation*}
 Indeed, we note that $C^{Hom,\lambda}_1(\Ac) =\langle e_1\ot e_2\rangle$, and that $b(e_1\ot e_1\ot e_2)=-2 e_1\ot e_2$.
  }
\end{example}

\subsubsection{The cyclic bicomplex for Hom-associative algebras}

We recall that the cyclic homology of a (not necessarily unital) associative algebra $A$ can also be defined as the total homology of a bicomplex \cite{Loday-book,Tsyg83}. We will now take a similar bicomplex approach to the cyclic homology for a Hom-associative algebra. We let
\begin{equation}\label{aux-norm-opt}
N:=1+ t_n +t_n^2 +...+t_n^n:C^{Hom}_n(\Ac)\lra C^{Hom}_n(\Ac)
\end{equation}
and the following lemma follows from the proof of Lemma \ref{lemma-b-b-prime-t}.

\begin{lemma}
The operators \eqref{aux-b-Hochschild}, \eqref{aux-b-prime} and \eqref{aux-norm-opt} satisfy the equality
\begin{equation*}
b'N=Nb.
\end{equation*}
\end{lemma}

As a result, we have the bicomplex $CC^{Hom}_{p,q}(\Ac):=\Ac^{\ot\,q+1}$ as follows:
\begin{align}\label{aux-cyclic-bicomplex}
\begin{CD}
\vdots @.\vdots @.\vdots @.\\
\Ac^{\ot\,3} @ <\Id-t_3 << \Ac^{\ot\,3}  @ < N << \Ac^{\ot\,3} @ <\Id-t_3 << \ldots  \\
@VVbV  @ VV- b'V @VV  b V  \\
\Ac^{\ot\,2} @ <\Id-t_2 << \Ac^{\ot\,2}  @ < N << \Ac^{\ot\,2} @ <\Id-t_2 << \ldots \\
@VV b V @VV- b'V @VV  bV  \\
\Ac  @<\Id-t_0 << \Ac  @ <  N << \Ac @ <\Id-t_0 << \ldots
\end{CD}
\end{align}

Let us denote the total homology of this bicomplex by $HC^{Hom}_\ast(\Ac)$. For more details of the total homology of a bicomplex (double complex) we refer the reader to \cite{Weib-book}.
Our task now is to show that this homology is  the cyclic homology of the Hom-associative algebra $\Ac$.

\begin{proposition}\label{prop-cyclic-homology}
For any Hom-associative algebra $\Ac$ over a field $k$ containing the rational numbers, we have $HC^{Hom}_\ast(\Ac) = HC^{Hom,\lambda}_\ast(\Ac)$.
\end{proposition}

\begin{proof}
We first observe that the rows of the bicomplex $CC^{Hom}_{p,q}(\Ac)$ are acyclic. Indeed, letting $\t:=\Id+t_n+2t_n^2+\ldots+nt_n^n$, we obtain at once that $N+\t(\Id-t_n)=(n+1)\Id$, and hence the operators $(1/n+1)\Id:\Ac^{\ot\,n+1}\lra \Ac^{\ot\,n+1}$ and $(1/n+1)\t:\Ac^{\ot\,n+1}\lra \Ac^{\ot\,n+1}$ define a homotopy
$$
\xymatrix{
\Ac^{\ot\,n+1}\ar[d]^{\Id}\ar[dr]^{\frac{\t}{n+1}} & \ar[l]_{\Id-t_n}\Ac^{\ot\,n+1}\ar[d]^{\Id}\ar[dr]^{\frac{\Id}{n+1}} & \ar[l]_{N}\Ac^{\ot\,n+1}\ar[d]^{\Id}\ar[dr]^{\frac{\t}{n+1}} & \ar[l]_{\Id-t_n}\Ac^{\ot\,n+1}\ar[d]^{\Id} \\
\Ac^{\ot\,n+1} & \ar[l]^{\Id-t_n}\Ac^{\ot\,n+1} & \ar[l]^{N}\Ac^{\ot\,n+1} & \ar[l]^{\Id-t_n}\Ac^{\ot\,n+1}
}
$$
on the rows of the bicomplex \eqref{aux-cyclic-bicomplex} from $\Id:\Ac^{\ot\,n+1}\lra \Ac^{\ot\,n+1}$ to $0:\Ac^{\ot\,n+1}\lra \Ac^{\ot\,n+1}$. As a result of the acyclicity of the rows, we obtain
\begin{equation*}
E^1_{p,q}=\left\{\begin{array}{cc}
                   0, & {\rm if}\,\,p>0 \\
                   C^{Hom,\lambda}_q, & {\rm if} \,\,p=0
                 \end{array}
\right.
\end{equation*}
on the $E^1$-page of the spectral sequence associated to the filtration of the bicomplex \eqref{aux-cyclic-bicomplex} via columns. Therefore
\begin{equation*}
E^2_{p,q}=\left\{\begin{array}{cc}
                   0, & {\rm if}\,\,p>0 \\
                   H^{Hom,\lambda}_q, & {\rm if} \,\,p=0.
                 \end{array}
\right.
\end{equation*}
Thus the claim follows since the spectral sequence degenerates at this level.
\end{proof}

One can extend the bicomplex \eqref{aux-cyclic-bicomplex}  to
\begin{align}
\begin{CD}
\vdots @.\vdots @.\vdots @.\\
\ldots \Ac^{\ot\,3} @ <\Id-t_2 << \Ac^{\ot\,3}  @ < N << \Ac^{\ot\,3} @ <\Id-t_2 << \ldots  \\
@VVbV  @ VV- b'V @VV  b V  \\
\ldots \Ac^{\ot\,2} @ <\Id-t_1 << \Ac^{\ot\,2}  @ < N << \Ac^{\ot\,2} @ <\Id-t_1 << \ldots \\
@VV b V @VV- b'V @VV  bV  \\
\ldots \Ac  @<\Id-t_0 << \Ac  @ <  N << \Ac @ <\Id-t_0 << \ldots
\end{CD}
\end{align}

The homology of the total complex of  this bicomplex is called the periodic cyclic homology of the Hom-associative algebra $\Ac$, and is denoted by $HP^{Hom}_*(\Ac)$. Just as in the associative case, we have only two periodic cyclic groups, the even $HP^{Hom}_0(\Ac)$, and the odd $HP^{Hom}_1(\Ac)$.

\begin{remark}{\rm
The cyclic homology of a non-unital Hom-associative algebra $\Ac$ can not be computed by the Connes' $(B,b)$-bicomplex \cite{Loday-book}, as the operator $B$ requires the unit. If, however, $\Ac$ is unital, then we can define a cyclic module structure on $C_\ast^{Hom}(\Ac)$ by the faces
\begin{align*}
& \d_i:C_n^{Hom}(\Ac) \lra C_{n-1}^{Hom}(\Ac),\qquad 0 \leq i \leq n\\
&\delta_i(a_0 \odots  a_n) = \alpha(a_0) \odots  a_i a_{i+1}\odots\alpha( a_n), \quad {0 \leq i < n}\\
&\delta_n(a_0 \odots  a_n)=   a_n a_0\otimes \alpha(a_1)\odots \alpha(a_{n-1}),
\end{align*}
the degeneracies
\begin{align*}
& C_n^{Hom}(\Ac) \lra C_{n+1}^{Hom}(\Ac), \qquad 0 \leq i \leq n\\
&\sigma_j (a_0\odots a_n)= a_0\odots a_j \otimes 1 \otimes a_{j+1}\odots a_n, \\
\end{align*}
and the cyclic operator
\begin{align*}
& t_n:C_n^{Hom}(\Ac) \lra C_{n}^{Hom}(\Ac),\\
&t_n (a_0\odots a_n)=a_n\otimes a_0\odots a_{n-1},
\end{align*}

and hence we can associate a $(b,B)$-bicomplex as follows
$$\begin{CD}
\vdots @.\vdots @.\vdots \\
C_2^{Hom}(\Ac) @<B<< C_1^{Hom}(\Ac) @<B<< C_0^{Hom}(\Ac) \\
@VbVV @VbVV\\
C_1^{Hom}(\Ac)@<B<<C_0^{Hom}(\Ac)\\
@VbVV\\
C_0^{Hom}(\Ac)
\end{CD} $$
where $B=(1-\lambda) s N   ,$ and $s= \tau_{n+1}\sigma_n:C_n \ra C_{n+1} .$
 We note that the total homology of the $(b,B)$-bicomplex is the same as the homology of the cyclic bicomplex \eqref{aux-cyclic-bicomplex}.
}
\end{remark}

We conclude this subsection by a short note on the functoriality of the Hom-associative extensions of the Hochschild and the cyclic homology theories.

\begin{proposition}
 Let $(\Ac, \a_{\Ac})$ and $(\Bc, \a_{\Bc})$ be two Hom-associative algebras. Then a morphism $f: \Ac\longrightarrow \Bc$ of Hom-associative algebras induces the maps
\begin{equation*}
f': HH^{Hom}_n(\Ac)\longrightarrow HH^{Hom}_n(\Bc), \qquad f'':HC^{Hom}_n(\Ac)\longrightarrow HC^{Hom}_n(\Bc),
\end{equation*}
 given by $a_0\ot \cdots \ot a_n\mapsto f(a_0)\ot \cdots \ot f(a_n)$. Furthermore if $\Ac$ and $\Bc$ are isomorphic as Hom-associative algebras then $HH^{Hom}_n(\Ac)\cong HH^{Hom}_n(\Bc)$ and $HC^{Hom}_n(\Ac)\cong HC^{Hom}_n(\Bc)$.
\end{proposition}

\begin{proof}
The claim follows from the fact that the induced maps commutes with the Hochschild boundary map, and that they respect the cyclic group action on the Hochschild complex.
\end{proof}

\section{Cohomology of Hom-associative algebras}

In this section we define the Hochschild and cyclic cohomologies of Hom-associative algebras.

\subsection{Hochschild cohomology of a Hom-associative\\ algebra}

In this subsection we introduce the Hochschild cohomology of a Hom-associative algebra $\Ac$ with coefficients in a dual $\Ac$-bimodule.

\begin{theorem}
  Let $(\Ac,  \alpha)$ be a Hom-associative algebra and $(V, \beta)$ be a dual $\Ac$-bimodule. Let $C^n_{Hom}(\Ac, V)$ be the space of all $k$-linear maps $\varphi: \Ac^{\ot n}\longrightarrow V$. Then
\begin{equation*}
C_{Hom}^\ast(\Ac, V)=\bigoplus_{n\geq0} C_{Hom}^n(\Ac, V),
\end{equation*}
with the operators
\begin{align}\label{aux-cosimplisial-structure-vp}
\begin{split}
&\d_0\varphi(a_1\odots a_{n+1})=a_1\cdot \varphi(\a(a_2)\odots \a(a_{n+1}))\\
&\d_i\varphi(a_1\odots a_{n+1})=\b(\varphi(\a(a_1)\odots a_i a_{i+1}\odots \a(a_{n+1}))), ~~ 1\leq i \leq n\\
&\d_{n+1}\varphi(a_1\odots a_{n+1})= \varphi(\a(a_1)\odots \a(a_{n}))\cdot a_{n+1}\\
\end{split}
\end{align}
is a pre-cosimplicial module.
\end{theorem}

\begin{proof}
We show that $\delta_i \delta_j= \delta_j \delta_{i-1}$ for $0\leq j< i \leq n-1$. Let us first show that $\d_1\d_0=\d_0\d_0$. Indeed,
\begin{align*}
\d_0(\d_0\varphi)(a_1\odots a_{n+2}) &=a_1\cdot \d_0\varphi(\a(a_2)\odots \a(a_{n+2}))\\
&= a_1\cdot (\a(a_2)\cdot\varphi(\a^2(a_3)\odots \a^2(a_{n+2}))) \\
& =\b((a_1a_2)\cdot\vp(\a^2(a_3)\odots \a^2(a_{n+2}))) \\
&= \b(\d_0\vp(a_1a_2\ot \a(a_3)\odots \a(a_{n+2})))\\
&=\d_1\d_0\vp(a_1\odots a_{n+2}).
\end{align*}
We used the  left dual module property   in the third equality. One can similarly use the right dual module property to show that $\delta_{n+1} \delta_n= \delta_n \delta_{n}$.
The following demonstrates that $\d_{n+1}\d_0=\d_0\d_{n}$. We have
\begin{align*}
  \d_{n+1}\d_0\varphi(a_1\odots a_{n+2})  &=\d_0\varphi(\a(a_1)\odots \a(a_{n+1}))\cdot a_{n+2}\\
  &=(\a(a_1)\cdot \varphi(\a^2(a_2)\odots \a^2(a_{n+1})))\cdot a_{n+2}\\
  &=a_1 \cdot(\varphi(\a^2(a_2)\odots \a^2(a_{n+1}))\cdot \a(a_{n+2}))\\
  &=a_1\cdot \d_n\varphi(\a(a_2)\odots \a(a_{n+2}))\\
  &=\d_0\d_n\varphi(a_1\odots a_{n+2}).
\end{align*}
The relation $ \delta_{j+1} \delta_j= \delta_j \delta_{j}$ follows from the Hom-associativity of $\Ac$. The rest of the relations are similar.
\end{proof}

The cohomology of the complex $(C_{Hom}^\ast(\Ac, V),b)$ is called the Hochschild cohomology of the Hom-associative algebra $\Ac$ with coefficients in a dual $\Ac$-bimodule $V$, and is denoted by $H_{Hom}^\ast(\Ac, V)$.

In particular, if $V$ is an $\Ac$-bimodule, then $(C_{Hom}^\ast(\Ac, V^\ast),b)$ is a differential complex with
\begin{align*}
b:C_{Hom}^n(\Ac, V^\ast)&\lra C_{Hom}^{n+1}(\Ac, V^\ast), \\
b\varphi(a_1\odots a_{n+1})&= a_1 \cdot \varphi(\a(a_2)\odots \a(a_{n+1}))\\
&\quad+\sum_{i=1}^n (-1)^{i}\varphi(\a(a_1)\odots a_ia_{i+1}\odots \a(a_{n+1}))\\
&\quad+\varphi(\a(a_1)\odots \a(a_n))\cdot a_{n+1}.
\end{align*}

Moreover, we have the following result.

\begin{proposition}\label{relation-Hoch}
For any Hom-associative algebra  $(\Ac,\mu,  \alpha)$, we have
\begin{equation*}
  H_n^{Hom}(\Ac, \Ac)^*\cong H^n_{Hom}(\Ac, \Ac^*).
\end{equation*}
\end{proposition}

\begin{proof}
The claim follows at once by observing that the isomorphisms
\begin{equation*}
  \Hom(M\ot \Ac^{\ot n}, k)\cong \Hom(\Ac^{\ot n}, \Hom(M, k))
\end{equation*}
are compatible with the differentials when $M=\Ac$. In fact the diagram
\begin{align*}
\xymatrix{
C^n_{Hom}(\Ac, \Ac^*)     \ar[r]^{\,\,\,\,\,f} \ar[d]_{d^n} &  C^{Hom}_n(\Ac, \Ac)^*       \ar[d]^{d_n^*}\\
C^{n+1}_{Hom}(\Ac, \Ac^*)  \ar[r]_{f} & C^{Hom}_{n+1}(\Ac, \Ac)^*
}
\end{align*}
is commutative where $f:C^n_{Hom}(\Ac, \Ac^*)\longrightarrow C^{Hom}_n(\Ac, \Ac)^*$ is given by $$f(\varphi)(a_0 \odots a_n)=\varphi(a_1\odots a_n)(a_0).$$
\end{proof}

\subsection{Cyclic cohomology of a Hom-associative algebra}

In this subsection we introduce the cyclic cohomology for Hom-associative algebras via two different but equivalent methods.

\subsubsection{Cyclic cohomology as a kernel}

In this subsection we generalize the cyclic cohomology to the setting of Hom-associative algebras along the lines of \cite{Conn85}. Recalling that $V=\Ac$ is naturally an $\Ac$-bimodule, we will consider the Hochschild cohomology of $\Ac$ with coefficients in the dual $\Ac$-bimodule $\Ac^\ast$.
Identifying $\vp\in C^n(A,A^\ast)$ with
\begin{equation*}\label{aux-identification}
\phi:A^{\ot\,n+1}\lra k,\qquad \phi(a_0\ot \a_1\odots a_n):=\vp(a_1 \odots a_n)(a_0),
\end{equation*}
 the coboundary map corresponds to
\begin{align*}
b:C^n(\Ac,\Ac^\ast)&\lra C^{n+1}(\Ac,\Ac^\ast),\\
b\phi(a_0\odots a_{n+1})&=\phi(a_0a_1\ot \a(a_2) \odots \a(a_{n+1}))\\
&\quad+\sum_{j=1}^n (-1)^j\phi(\a(a_0) \odots a_ja_{j+1} \odots \a(a_{n+1}))\\
&\quad+(-1)^{n+1} \phi(a_{n+1}a_0\ot \a(a_1) \odots \a(a_n)).
\end{align*}
Moreover, the pre-cosimplicial structure is translated into
\begin{align*}
\begin{split}
    &\d_0\phi(a_0 \odots a_n)= \phi(a_0a_1\ot \a(a_2) \odots \a(a_{n}))\\
    &\d_i\phi(a_0 \odots a_n)=\phi(\a(a_0) \odots a_i a_{i+1} \odots \a(a_{n})), ~~ 1\leq i \leq n-1\\
    &\d_{n}\phi(a_0 \odots a_n)=\phi(a_{n}a_0\ot \a(a_1) \odots \a(a_n)).
\end{split}
\end{align*}
These maps are dual to the ones for $C_n^{Hom}(\Ac, \Ac)$.
We define
\begin{align}\notag
&  t_n: C^n_{Hom}(\Ac, \Ac^*)\longrightarrow C^n_{Hom}(\Ac, \Ac^*) \\
& t_n\phi(a_0\ot a_1\odots a_n):=(-1)^n\phi(a_n\ot a_0\ot a_1\odots a_{n-1}),\label{aux-tau-cohom}
\end{align}
and we set
\begin{align*}
\begin{split}
C^n_{\lambda,Hom}(\Ac,\Ac^\ast) &= \ker(\Id-t_n)\\
&= \{\phi\in C^n_{Hom}(\Ac,\Ac^\ast)\,|\,\phi(a_0\ot a_1\odots a_n)\\&\qquad\qquad\qquad\qquad\qquad=(-1)^n\phi(a_n\ot a_0\odots a_{n-1})\}.
\end{split}
\end{align*}
The following lemma shows that $C^\ast_{\lambda,Hom}(\Ac,\Ac^\ast)$ is a subcomplex of the Hochschild complex $C^\ast_{Hom}(\Ac,\Ac^\ast)$.

\begin{lemma}
If $\phi \in C^n_{\lambda,Hom}(\Ac,\Ac^\ast)$, then $b\phi \in C^{n+1}_{\lambda,Hom}(\Ac,\Ac^\ast)$.
\end{lemma}

\begin{proof}
We will show that
\begin{equation*}\label{aux-pre-cosimplicial}
t_n\d_i  = -\d_{i-1}t_{n-1}, \quad 1 \leq i \leq n, \qquad  \quad \tau_{n}\d_0 =\d_{n}.
\end{equation*}
First we check that
\begin{align*}
\begin{split}
t_n\d_0\phi(a_0 \odots a_n) &= \d_0\phi(a_n\ot a_0 \odots a_{n-1}) \\
&= \phi(a_na_0\ot \a(a_1) \odots \a(a_{n-1}))\\
& = \d_n\phi(a_0 \odots a_n),
\end{split}
\end{align*}
and next we observe that
\begin{align*}
\begin{split}
 t_{n+1}\d_i\phi(a_0 \odots a_{n+1}) &= \d_i\phi(a_{n+1}\ot a_0 \odots a_n) \\
&= \phi(\a(a_{n+1})\ot \a(a_0) \odots a_{i-1}a_i\odots \a(a_n)) \\
& =-t_n\phi(\a(a_0) \odots a_{i-1}a_i\odots \a(a_{n+1})) \\
& =-\d_{i-1}(t_n\vp)(a_0 \odots a_{n+1}).
\end{split}
\end{align*}
\end{proof}

We call the homology of the complex $(C_{\lambda,Hom}^\ast(\Ac),b)$ the cyclic cohomology of the Hom-associative algebra $\Ac$, and we denote it by $HC^\ast_{Hom,\lambda}(\Ac)$.

\subsubsection{The cocyclic bicomplex for Hom-associative algebras}

In this subsection we dualize the bicomplex \eqref{aux-cyclic-bicomplex} to obtain a similar bicomplex whose total homology is the cyclic cohomology of a Hom-associative algebra $\Ac$.


Similar to the homology case, we set
\begin{align}\label{aux-b-prime-cohom}
\begin{split}
& b':C_{Hom}^n(\Ac, \Ac^*)\lra C_{Hom}^{n+1}(\Ac, \Ac^*),\\
& b'(\varphi)(a_0 \odots a_{n+1})= \sum_{i=0}^{n-1} (-1)^i\varphi(\a(a_0)\odots a_ia_{i+1}\odots \a(a_n)),
\end{split}
\end{align}
and
\begin{equation}\label{aux-norm-opt-cohom}
N:=\Id+t_n+\ldots+t_n^n:C_{Hom}^n(\Ac)\lra C_{Hom}^n(\Ac).
\end{equation}

\begin{lemma}
The operators \eqref{aux-tau-cohom}, \eqref{aux-b-prime-cohom} and \eqref{aux-norm-opt-cohom} satisfy the identities
\begin{equation*}
(\Id-t_{n+1})b=b'(\Id-t_n), \quad (\Id-t_n)N=N(\Id-t_n)=0,\quad \text{and} \quad Nb'=bN.
\end{equation*}
\end{lemma}

Consequently, we have the bicomplex
\begin{equation*}
CC^{p,q}_{Hom}(\Ac):=\{\phi:\Ac^{\ot\,q+1}\lra k\mid \phi \,\,{\rm is\,\, linear}\}
\end{equation*}
as follows. Letting $C^q:=CC^{p,q}_{Hom}(\Ac)$,
\begin{align}\label{aux-cocyclic-bicomplex}
\begin{CD}
\vdots @.\vdots @.\vdots @.\\
C^2 @ >\Id-t_2 >> C^2  @ > N >> C^2 @ >\Id-t_2 >> \cdots  \\
@AAbA  @ AA- b'A @AA  b A  \\
C^1 @ >\Id-t_1 >> C^1  @ > N >> C^1 @ >\Id-t_1 >> \cdots \\
@AA  b A @AA- b'A @AA  bA  \\
C^0  @>\Id-t_0 >> C^0  @ >  N >> C^0 @ >\Id-t_0 >> \cdots
\end{CD}
\end{align}

We denote the total cohomology of the bicomplex \eqref{aux-cocyclic-bicomplex} by $HC^*_{\Hom}(\Ac)$. By the next proposition we show that the total cohomology is the same as the cyclic cohomology $HC^\ast_{Hom,\lambda}(\Ac)$ of the Hom-associative algebra $\Ac$.

\begin{proposition}
For any Hom-associative algebra $\Ac$ over a field $k$ containing the rational numbers, $HC^\ast_{Hom,\lambda}(\Ac)=HC^*_{Hom}(\Ac)$.
\end{proposition}

\begin{proof}
A similar argument as in the proof of Proposition \ref{prop-cyclic-homology} yields the acyclicity of the rows of the bicomplex \eqref{aux-cocyclic-bicomplex}. Thus, considering the spectral sequence associated to the filtration of the bicomplex \eqref{aux-cocyclic-bicomplex} via columns, we obtain
\begin{equation*}
E^1_{p,q}=\left\{\begin{array}{cc}
                   0, & {\rm if}\,\,p>0 \\
                   C_{Hom,\lambda}^q, & {\rm if} \,\,p=0
                 \end{array}
\right.
\end{equation*}
on the $E^1$-page. Therefore
\begin{equation*}
E^2_{p,q}=\left\{\begin{array}{cc}
                   0, & {\rm if}\,\,p>0 \\
                   H_{Hom,\lambda}^q, & {\rm if} \,\,p=0.
                 \end{array}
\right.
\end{equation*}
Then the claim follows since the spectral sequence degenerates at this level.
\end{proof}

In order to define the periodic cyclic cohomology we extend the bicomplex \eqref{aux-cocyclic-bicomplex} to
\begin{align*}\
\begin{CD}
\vdots @.\vdots @.\vdots @.\\
\cdots C^2 @ >\Id-t_2 >> C^2  @ > N >> C^2 @ >\Id-t_2 >> \cdots  \\
@AAbA  @ AA- b'A @AA  b A  \\
\cdots C^1 @ >\Id-t_1 >> C^1  @ > N >> C^1 @ >\Id-t_1 >> \cdots \\
@AA  b A @AA- b'A @AA  bA  \\
\cdots C^0  @>\Id-t_0 >> C^0  @ >  N >> C^0 @ >\Id-t_0 >> \cdots
\end{CD}
\end{align*}
and we call the homology of the total complex $CC^{p,q}_{Hom, per}(\Ac)$ of the extended bicomplex the periodic cyclic cohomology of the multiplicative Hom-associative algebra $\Ac$, and we denote it by $HP_{Hom}^\ast(\Ac)$. Just as the periodic cyclic homology, there are only two periodic cohomology groups, the even $HP_{Hom}^0(\Ac)$ and the odd $HP_{Hom}^1(\Ac)$.

\begin{remark}
{\rm
Similar to the homology case,  the Connes' $(B,b)$-bicomplex can not be defined for non-unital Hom-associative algebras. However, if  $(\Ac,\mu, \alpha)$ is a unital  Hom-associative algebra, via the cofaces
   \begin{align*}
&\delta_i: C^n_{Hom}(\Ac)\lra C^{n+1}_{Hom}(\Ac), \qquad {0 \leq i \leq n+1}\\
&\delta_i\varphi(a_0 \odots  a_{n+1}) = \varphi(\a(a_0)\odots a_i a_{i+1}\odots  \a(a_{n+1})), \quad {0 \leq i < n}\\
&\delta_{n+1}\varphi(a_0\odots a_{n+1})=    \varphi(a_{n+1} a_0\odots \a(a_1)\odots \a(a_n)),
\end{align*}
the codegeneracies,
   \begin{align*}
&\s_j: C^n_{Hom}(\Ac)\lra C^{n-1}_{Hom}(\Ac), \qquad {0 \leq i \leq n-1}\\
&\sigma_j \varphi(a_0\odots a_n)= \varphi(a_0\odots a_j \otimes 1 \otimes a_{j+1}\odots a_n),
\end{align*}
and the cyclic maps
   \begin{align*}
&t_n: C^n_{Hom}(\Ac)\lra C^n_{Hom}(\Ac),\\
&t_n \varphi(a_0\odots a_n)=\varphi(a_n\otimes a_0\odots a_{n-1}),
\end{align*}
we obtain a cocyclic module structure \cite{Loday-book}, via which we define the $(B,b)$-bicomplex.
}
\end{remark}

As for the functoriality of the cyclic cohomology for Hom-associative algebras, we have the following. The proof is similar to the case of homology, and hence is omitted.

\begin{proposition}
 Let $(\Ac, \a_{\Ac})$ and $(\Bc, \a_{\Bc})$ be two Hom-associative algebras. Then a morphism $f: \Ac\longrightarrow \Bc$ of Hom-associative algebras induces the maps
\begin{equation*}
f': HH_{Hom}^n(\Ac)\longrightarrow HH_{Hom}^n(\Bc), \qquad f'':HC_{Hom}^n(\Ac)\longrightarrow HC_{Hom}^n(\Bc),
\end{equation*}
 given by $a_0\ot \cdots \ot a_n\mapsto f(a_0)\ot \cdots \ot f(a_n)$. Furthermore if $\Ac$ and $\Bc$ are isomorphic as Hom-associative algebras then $HH_{Hom}^n(\Ac)\cong HH_{Hom}^n(\Bc)$ and $HC_{Hom}^n(\Ac)\cong HC_{Hom}^n(\Bc)$.
\end{proposition}

We conclude this subsection by the following examples to illustrate the elementary computations on cyclic cohomology.

\begin{example}{\rm
     Let $(\Ac,\mu,  \a)$ be a Hom-associative algebra. Then
\begin{equation*}
  H_{Hom}^{0}(\Ac, \Ac^*)= HC_{Hom}^0(\Ac).
\end{equation*}
More precisely,  the cyclic (or Hochschild) $0$-cocycles  of $\Ac$ are traces on $\Ac$, \ie $k$-linear maps    $\varphi: \Ac\longrightarrow k$ such that $\varphi(xy)=\varphi(yx)$, for all $x,y\in \Ac$.
}
\end{example}

\begin{example}{\rm
Let $(\Ac,\mu,  \a)$ be a  Hom-associative algebra. Then the cyclic $1$-cocycles $\varphi: \Ac\ot \Ac\longrightarrow k$ are the Hochschild $1$-cocycles
\begin{equation*}
  \varphi(xy\ot \a(z))-\varphi(\a(x)\ot yz)+\varphi(zx\ot \a(y))=0,
\end{equation*}
which are cyclic
\begin{equation*}
  \varphi(x\ot y)=-\varphi(y\ot x),
\end{equation*}
for all $x,y,z\in \Ac$.
}
\end{example}

On the next example we discuss how to construct a cyclic $1$-cocycle out of a cyclic $0$-cocyle and an $\a$-derivation.

\begin{example}
  {\rm
Let $(\Ac,\mu,  \a)$ be a multiplicative Hom-associative algebra,  $\rho: \Ac\longrightarrow \Ac$ a twisted $\a$-derivation, \ie $\rho(ab)=\rho(a)b+ a\rho(b)$, where $\a\rho=\rho\a=\rho$.

Let also $tr: \Ac \longrightarrow k$ be a trace map, \ie $tr(ab)=tr(ba)$ and $tr(\rho(a))=0$.

We will show that  the map
\begin{equation*}
\varphi: \Ac\ot \Ac\longrightarrow k,\qquad \varphi(a,b)=tr(a\rho(b))
\end{equation*}
is a cyclic $1$-cocycle. We first show that $\varphi$ is a Hochshcild $1$-cocycle. Indeed,
\begin{align*}
 \varphi&(ab, \alpha(c))-\varphi(\alpha(a), bc)+ \varphi(ca, \alpha(b))\\
  &=tr(ab\rho( \alpha(c)))-tr(\alpha(a)\rho( bc))+ tr(ca\rho( \alpha(b)))\\
  &=tr(ab\rho( \alpha(c)))-tr(\alpha(a)\rho(b ) c))\\
  &=-tr(\a(a) b \rho(c))+ tr(ca\rho( \alpha(b))) = 0.
 \end{align*}
On the other hand,
\begin{equation*}
\vp(a,b)=tr(a\rho(b))=tr(\rho(ab)-b\rho(a)) = -tr(b\rho(a)) = -\vp(b,a),
 \end{equation*}
hence $\varphi$ is cyclic.
}
\end{example}

The final example is about the relation between the cyclic homologies of an associative algebra, and a Hom-asoociative algebra attached to it.

\begin{example}{\rm
Let $A$ be an associative algebra and $\alpha: A\longrightarrow A$ be  an algebra map where $\alpha^2=\alpha$. We then have the  non-unital Hom-associative algebra $(A_{\alpha}, \mu_{\alpha}, \alpha)$ of Example \ref{twist}.

Letting $\Ac:=A_\a$, we define the map
  \begin{equation}\label{aux-map-xi}
    \xi: C_{\lambda}^n(A)\longrightarrow C_{\lambda}^n(\Ac), \quad \varphi\longmapsto \varphi_{\alpha}
  \end{equation}
  where
\begin{equation*}
\varphi_{\alpha}(a_0\odots a_n)= \varphi(\alpha(a_0) \odots \alpha(a_n)).
\end{equation*}
Then it follows from
  \begin{align*}
   \delta_i \xi\varphi(a_0 \odots a_n)   &=\xi\varphi(\alpha(a_0) \odots \alpha(a_i a_{i+1}) \odots \alpha(a_n))\\
   &=\varphi(\alpha^2(a_0) \odots \alpha^2(a_i)\alpha^2( a_{i+1}) \odots \alpha^2(a_n))\\
   &=\varphi(\alpha(a_0) \odots \alpha(a_i)\alpha( a_{i+1}) \odots \alpha(a_n))\\
   &=(\delta_i \varphi)(\alpha(a_0) \odots  \a(a_n))\\
   &=\xi \delta_i \varphi(a_0 \odots a_n)
  \end{align*}
that the map \eqref{aux-map-xi} commutes with the coboundary map of the Hochschild complex $C^\ast_{Hom}(\Ac,\Ac^\ast)$, and hence induces a map
\begin{equation*}
H^\ast(A, A^\ast)\longrightarrow H^\ast_{Hom}(\Ac,\Ac^\ast).
\end{equation*}
We note similarly that
  \begin{align*}
    \varphi_{\a}(a_0 \odots  a_n)&= \varphi(\a(a_0) \odots \a(a_n))\\
    &= (-1)^n\varphi(\a(a_n)\ot \a(a_0) \odots \a(a_{n-1}))\\
    &=(-1)^n \varphi_{\a}(a_n\ot a_0 \odots a_{n-1}),
  \end{align*}
that is, \eqref{aux-map-xi} induces a map
\begin{equation*}
HC^\ast(A)\longrightarrow HC_{Hom}^\ast(\Ac)
\end{equation*}
on the level of cyclic cohomologies.
  }
\end{example}

\medskip

\textbf{Acknowledgments}:
The authors wish to thank Donald Yau for his encouragement to  investigate  the cyclic homology theory for Hom-associative algebras. We also thank Aron Gohr for the Lemma \ref{lemma-unital}. The first author thanks  the Institut des Hautes \'{E}tudes Scientifiques for its hospitality and financial support during his visit. Finally, the authors would like to thank the referee for his/her constructive comments improving the paper.

\bibliographystyle{amsplain}
\bibliography{references}{}
\medskip

\noindent Department of Mathematics and Statistics,
University of Windsor, 401 Sunset Avenue, Windsor, Ontario N9B 3P4, Canada

\noindent\emph{E-mail address}:
\textbf{mhassan@uwindsor.ca}
\medskip

\noindent Department of Mathematics and Statistics,
University of Windsor, 401 Sunset Avenue, Windsor, Ontario N9B 3P4, Canada

\noindent\emph{E-mail address}:
\textbf{ishapiro@uwindsor.ca}
\medskip

\noindent

\noindent\emph{E-mail address}:
\textbf{serkansutlu@gmail.com}
\end{document}